\documentclass[12pt,a4paper]{amsart}

\usepackage{amsmath,amsfonts,amscd,amsthm,amssymb,latexsym,bbold}

\newtheorem{proposition}{Proposition}

\newtheorem{theorem}{Theorem}
\newtheorem{corollary}{Corollary}

\usepackage{hyperref}
\usepackage[in]{fullpage}
\usepackage{fourier}

\DeclareMathAlphabet{\pazocal}{OMS}{zplm}{m}{n}

\def\calO{\pazocal{O}}

\def\calS{\pazocal{S}}
\def\calT{\pazocal{T}}

\def\calX{\pazocal{X}}

\DeclareMathAlphabet{\mathbbold}{U}{bbold}{m}{n}

\DeclareMathOperator{\Com}{\mathsf{Com}}
\DeclareMathOperator{\CDi}{\mathsf{CD}}
\DeclareMathOperator{\Nov}{\mathsf{Nov}}

\def\k{\mathbbold{k}}

\title{Polynomial identities in Novikov algebras}

\begin{document}

\author{Vladimir Dotsenko}

\address{ 
Institut de Recherche Math\'ematique Avanc\'ee, UMR 7501, Universit\'e de Strasbourg et CNRS, 7 rue Ren\'e-Descartes, 67000 Strasbourg CEDEX, France}

\email{vdotsenko@unistra.fr}

\author{Nurlan Ismailov}

\address{Astana IT University, Astana, 010000, Kazakhstan}

\email{nurlan.ismail@gmail.com}

\author{Ualbai Umirbaev}

\address{Department of Mathematics, Wayne State University, Detroit, MI 48202, USA; Department of Mathematics, Al-Farabi Kazakh National University, Almaty, 050040, Kazakhstan; Institute of Mathematics and Mathematical Modeling, Almaty, 050010, Kazakhstan,}

\email{umirbaev@wayne.edu}

\dedicatory{To Ivan Pavlovich Shestakov on the occasion of his 75th birthday}
 
\date{}

\begin{abstract}
In this paper, we study Novikov algebras satisfying nontrivial identities. We show that a Novikov algebra over a field of zero characteristic that satisfies a nontrivial identity satisfies some unexpected ``universal'' identities, in particular, right associator nilpotence, and right nilpotence of the commutator ideal. This, in particular, implies that 
a Novikov algebra over a field of zero characteristic satisfies a nontrivial identity if and only if it is Lie-solvable. We also establish that any system of identities of Novikov algebras over a field of zero characteristic follows from finitely many of them, and that the same holds over any field for multilinear Novikov identities. Some analogous simpler statements are also proved for commutative differential algebras.
\end{abstract}



\maketitle

\section{Introduction}

A vector space $N$ over a field $\k$ equipped with a product $x,y\mapsto x\circ y$ is called a {\em Novikov algebra} if the following  identities hold for all $x,y,z\in N$:
\begin{gather*}
(x\circ y)\circ z-x\circ(y\circ z)=(y\circ x)\circ z-y\circ (x\circ z),\\
(x\circ y)\circ z=(x\circ z)\circ y.
\end{gather*}

The term ``Novikov algebra'' was coined by Osborn \cite{MR1163779}. In fact, the identities of Novikov algebras seem to have first appeared in the study of Hamiltonian operators in the formal calculus of variations by Gelfand and Dorfman \cite{GD79}, and then rediscovered by Balinskii and Novikov in the context of classification of linear Poisson brackets of hydrodynamical type~\cite{BN85}. 

In this paper, we advance in PI-theory for Novikov algebras, that is, study two aspects of Novikov algebras satisfying additional identities. We prove results on solvability and nilpotence of such algebras, and we establish existence of a finite basis of identities of any given algebra. Throughout the paper, we present our results using the language of operads, ideals in operads, and modules over operads. We believe that exploring the advantages of this language when dealing with polynomial identities is a very timely task, and we hope that our work will help ring theorists to get accustomed to this language.

\subsection{Nilpotence and solvability}
The questions of nilpotence and solvability of Novikov algebras have been studied quite extensively, starting from the result of Zelmanov \cite{Zel} from 1987 who proved that if $N$ is a finite dimensional right nilpotent Novikov algebra then $N^2$ is nilpotent. In the following years, Filippov \cite{Fil01} proved that any left-nil Novikov algebra of bounded index over a field of characteristic zero is nilpotent, and Dzhumadildaev and Tulenbaev \cite{DzT} proved that any right-nil Novikov algebra of bounded index $n$ is right nilpotent if the characteristic of the ground field $\k$ is $0$ or $p>n$. More recently , Shestakov and Zhang proved \cite{ShZh} that a Novikov algebra $N$ over any field is solvable if and only if it is right nilpotent, and that both conditions are also equivalent to the nilpotency of $N^2$,
Zhelyabin and the third author of the present paper proved in \cite{UZh21} solvability of a Novikov algebra graded by a finite abelian group is solvable if the homogeneous component of the unit is solvable if the characteristic of the ground field does not divide the order of the group, and Tulenbaev, Zhelyabin and the third author of the present paper proved that \cite{TUZ21} a Novikov algebra $N$ over a field of characteristic different from~$2$ is solvable as a Lie algebra if and only if the commutator ideal $[N,N]$ is right nilpotent.  

Our work started with the discovery of the following surprising property of Novikov algebras over a field of zero characteristic: if a Novikov algebra $N$ satisfies a nontrivial identity, then this algebra is right associator nilpotent, that is $(N,N,N)^R_p=0$ for some $p$, where $(N,N,N)^R_1=(N,N,N)$ is the linear span of all associators $(x,y,z)=(x\circ y)\circ z-x\circ(y\circ z)$, and 
$(N,N,N)^R_{p+1}=((N,N,N)^R_p,N,N)$. This result may be viewed as a counterpart for Novikov algebras of the associator nilpotency theorem for alternative algebras proved by Shestakov \cite{MR674178}. Using the right associator nilpotency, we prove that a Novikov algebra over a field of zero characteristic satisfies a nontrivial identity if and only if it is Lie-solvable: the Lie bracket on it given by the formula $[a,b]=a\circ b-b\circ a$ defines a solvable Lie algebra. 

\subsection{Finite basis of identities}
Given a variety (or, to use less ambiguous terminology, an equational class) of algebras, one may ask if any subvariety of it can be defined by finitely many identities. This is usually referred to as the Specht property of a variety, as a tribute to foundational work of Specht on varieties of associative algebras \cite{MR35274}. Over a field of zero characteristic, it was established by Kemer \cite{MR937115,MR1108620} that the variety of associative algebras has the Specht property; over a field of characteristic $p>0$, there exist subvarieties of the variety of associative algebras with an infinite basis of identities, see \cite{MR1773251,MR1799541,MR1799533}. The conjecture of Kemer \cite{MR1244616} that a subvariety of the variety of associative algebras over a field of characteristic $p>0$ defined by multilinear identities can be defined by finitely many of them still remains an open problem. 

Some generalizations of Kemer's approach to associative algebras were studied by Belov \cite{MR2643374} in the nonassociative case; whenever his results are applicable to some variety, relatively free algebras of that variety have rational Hilbert series, which rules out the variety of Novikov algebras, where the Hilbert series of the free one-generated Novikov algebra, computed by Dzhumadildaev and the second author of the present paper \cite{MR3241181}, is the generating series for the partition function, which is not rational. Another hint as to why the methods of Kemer are unlikely to work in the case of Novikov algebras comes from remarking that his approach uses in a meaningful way the structure theory for finite-dimensional associative algebras, and in the case of Novikov algebras, a theorem of Zelmanov \cite{Zel} asserts that a simple Novikov algebra over a field of zero characteristic is a field, which leaves no hope for using simple Novikov algebras in order to understand identities.

The identities of Novikov algebras also resemble the identities of the so called bicommutative algebras, for which 
\begin{gather*}
x\circ(y\circ z)=y\circ (x\circ z),\\
(x\circ y)\circ z=(x\circ z)\circ y.
\end{gather*}
The Specht property of this latter variety was established by Drensky and Zhakhayev~\cite{MR3758517} over any ground field using Higman's Lemma \cite{MR49867}, which is perhaps the oldest known approach to the Specht property, see, for instance, its application to some varieties of Lie algebras by Bryant and Vaughan-Lee \cite{MR297825} and by Sheina  \cite{MR0409584}, and to some varieties of alternative algebras by the third author of the present paper \cite{MR816580}. The variety of Novikov algebras appears to be somewhat richer than that of bicommutative algebras: the combinatorics of the basis elements is more intricate in the Novikov case, and the class of polynomial representations of the general linear groups that appear in free Novikov algebras is also much larger than that of free bicommutative algebras \cite{MR3241181,MR2840273}. Nevertheless, we found an indirect way to apply Higman's Lemma and managed to prove the Specht property over a field of zero characteristic; in fact, we prove the stronger finite basis property for multilinear identities over a field of any characteristic. 

\subsection{Novikov algebras and differential algebras}
Our work relies in a very meaningful way on the connection that exists between Novikov algebras and commutative associative differential algebras. In fact, already in \cite{GD79} a general construction of Novikov algebras is given (note that I.~M.~Gelfand and I.~Ja.~Dorfman in \cite{GD79} attribute this construction to S.~I.~Gelfand). Namely, let $A$ be a commutative associative algebra, and let $D$ be a derivation of~$A$. One may define a new  multiplication $\circ$ on $A$ by setting 
\[
x\circ y=x D(y). 
\] 
It turns out that the property of $D$ being a derivation guarantees that both identities of Novikov algebras hold for this product. Dzhumadildaev and L\"ofwall proved in \cite{DzhL} that this construction can be used to embed each free Novikov algebra into the free commutative associative differential algebra on the same generating set. Moreover, one may regard this construction as a functor from the category of commutative associative differential algebras to the category of Novikov algebras,
a multilplication changing functor in the sense of Mikhalev and Shestakov \cite{MR3169596}, and consider its left adjoint, the \emph{universal differential enveloping algebra} of a Novikov algebra. It was proved by Bokut, Chen, and Zhang \cite{BCZ18} that every Novikov algebra embeds into its universal differential enveloping algebra, so one can faithfully represent Novikov algebras using the construction of S.~I.~Gelfand. 

\subsection{Structure of the paper}
The paper is organized as follows. In Section \ref{sec:recoll}, we recall the necessary definitions and results.  
In Section \ref{sec:toy}, we discuss consequences of differential identities of commutative associative algebras, which allows us to prepare the reader for the proofs of our main theorems by considering an interesting toy model. We establish two new results in that case. First, if a commutative associative differential algebra  satisfies a nontrivial identity, then it satisfies an identity of the form
 \[
D(x_1)\cdots D(x_n)=0
 \]
for some $n$. Second, the variety of commutative associative differential algebras over a field of zero characteristic has the Specht property, moreover, every system of multilinear differential identities over any field is finitely based.  
In Section \ref{sec:IdNov}, we prove the main results of the paper. We start from an auxiliary result about consequences of differential identities where only Novikov operations may be used to derive consequences, which we then use to prove that over a field of zero characteristic a Novikov algebra satisfying a nontrivial identity is right associator nilpotent, that its commutator ideal is right nilpotent, and that a Novikov algebra satisfies a nontrivial identity if and only if it is Lie-solvable. We then prove that the variety of Novikov algebras over a field of zero characteristic has the Specht property moreover, that every system of multilinear Novikov identities over any field is finitely based.

We would like to dedicate this article to Ivan Pavlovich Shestakov on the occasion of his 75th birthday. His tremendous contributions to ring theory in the past fifty years have strongly influenced the state-of-the-art in the field as well as our personal tastes, and this article in particular is inspired by some of his work. 

\section{Recollections}\label{sec:recoll}

Throughout this paper, all algebras are defined over an arbitrary field $\k$; for some results it will be assumed of zero characteristic, in which case that will be explicitly stated.

\subsection{Operads} As indicated in the introduction, we use the language of the operad theory to state and prove our main results. In this section, we give a brief recollection of necessary definitions related to operads. Ring theorists wishing to get a short introduction to how operads encode varieties of algebras are invited to consult \cite{DU22,MR2533586}; readers interested in a more systematic approach are referred to the monographs \cite{BDo,MR2954392}. 

It is well known that over a field $\k$ of characteristic zero every identity in $\k$-algebras of any type is equivalent to a multilinear one. The notion of operad makes the most of that fact by studying multilinear identities in a highly structured way. Specifically, to a variety of $\k$-algebras $\mathfrak{M}$, one associates the datum 
 \[
\calO=\calO_\mathfrak{M}:=\{\calO(n)\}_{n\ge 1},
 \]
where $\calO(n)$ is the $S_n$-module of multihomogeneous elements of degree $1$ in each generator in the free algebra $F_\mathfrak{M}\langle x_1,\ldots,x_n\rangle$. 

The construction of a free $\mathfrak{M}$-algebra on a given vector space may be viewed as a functor from vector spaces to $\mathfrak{M}$-algebras, but also, in a more basic way, as a functor from the category of vector spaces to itself. In fact, this latter functor is of a very particular type: it has a structure of a \emph{monad}. Namely, if we apply that functor twice, considering $F_\mathfrak{M}\langle F_\mathfrak{M}\langle V\rangle\rangle$, we obtain the vector space of all possible substitution schemes of $\mathfrak{M}$-polynomials into each other. There is a canonical linear map
 \[
\tau_V\colon F_\mathfrak{M}\langle F_\mathfrak{M}\langle V\rangle\rangle\to F_\mathfrak{M}\langle V\rangle,
 \]
which says that for a substitution scheme, we can actually perform a substitution, and write an $\mathfrak{M}$-polynomial of $\mathfrak{M}$-polynomials as an $\mathfrak{M}$-polynomial. It is this map $\tau$ that gives our functor a monad structure, meaning that it is associative: if we apply our functor three times, forming the gigantic algebra
 \[
F_\mathfrak{M}\langle F_\mathfrak{M}\langle F_\mathfrak{M}\langle V\rangle\rangle\rangle,
 \]
there are two different maps to $F_\mathfrak{M}\langle V\rangle$, depending on the order of substitutions, and those give the same result. (Strictly speaking, to talk about a monad, one should also discuss the unitality, but the compatibility of the maps $\tau_V$ with the obvious embedding maps $\imath_V\colon V\to F_\mathfrak{M}\langle V\rangle$ is too trivial to spend time on it.) This can be restricted to multilinear elements, and it defines what is called an operad structure on the sequence $\{\calO(n)\}_{n\ge 1}$.

One key difference between the language of varieties and the language of operads may have already become apparent. Namely, in terms of varieties of algebras, properties like commutativity and associativity are on the same ground: both express certain identities in algebras. In terms of operads, commutativity of an operation is a symmetry type, a consequence of the fact that an operad is in particular a sequence of $S_n$-modules, while associativity of an operation is an identity: a relation between results of substitution of operations into one another. This distinction will be very important for us. Another remark is that, while to an expert of universal algebra, the notion of an operad seems like a minor variation of what has been long known under the name of \emph{clone} (moreover, the first ever paper about operads, that of Artamonov \cite{MR0237408}, introduced them under the name ``clone of multilinear operations'', as opposed to the catchy term ``operad'' coined by May \cite{MR0420610,MR2177746}), working in the $\k$-linear situation moves away from the cartesian context of clones, and allowing one to focus on multilinear operations, making many questions (and answers) much more combinatorially transparent.

A clean interpretation of how an operad structure formalizes the notion of substitutions of multilinear maps uses the notion of a linear species~\cite{MR2724388}, generalizing the combinatorial species of Joyal \cite{MR1629341,MR633783}. A \emph{linear species} is a contravariant functor from the groupoid of finite sets (the category whose objects are finite sets and whose morphisms are bijections) to the category of vector spaces. This definition is not easy to digest at a first glance, and a reader with intuition coming from varieties of algebras is invited to think of the value $\calS(I)$ of a linear species $\calS$ on a finite set $I$ as of the set of multilinear operations of type $\calS$ (accepting arguments from some vector space $V_1$ and assuming values in some vector space $V_2$) whose inputs are indexed by $I$. Sometimes, a ``skeletal definition'' is preferable: the category of linear species is equivalent to the category of symmetric sequences $\{\calS(n)\}_{n\ge 0}$, where each $\calS(n)$ is a right $S_n$-module, a morphism between the sequences $\calS_1$ and $\calS_2$ in this category is a sequence of $S_n$-equivariant maps $f_n\colon \calS_1(n)\to\calS_2(n)$. While this definition may seem more appealing, the functorial definition simplifies the definitions of operations on linear species: it is harder to comprehend the two following definitions skeletally.

The \emph{composition product} of linear species is defined by the formula 
 \[
(\calS_1\circ\calS_2)(I)
=\bigoplus_{n\ge 0}\calS_1(\{1,\ldots,n\})\otimes_{\k S_n}\left(\bigoplus_{I=I_1\sqcup \cdots\sqcup I_n}\calS_2(I_1)\otimes\cdots\otimes \calS_2(I_n)\right).
 \]
The linear species $\mathbbold{1}$ which vanishes on a finite set $I$ unless $|I|=1$, and whose value on $I=\{a\}$ is given by $\k a$ is the unit for the composition product: we have $\mathbbold{1}\circ\calS=\calS\circ\mathbbold{1}=\calS$.

Formally, a \emph{symmetric operad} is a monoid with respect to the composition product. It is just the multilinear version of substitution schemes of free algebras discussed above, but re-packaged in a certain way. The advantage is that existing intuition of monoids and modules over them, available in any monoidal category \cite{MR0354798}, can be used for studying varieties of algebras. In particular, one can talk about two-sided ideals of operads. A two-sided ideal of an operad is what one may use to form a quotient operad; in the language of varieties of algebras the equivalent notion is that of a T-ideal, that is an ideal of the free algebra that is additionally closed under all endomorphisms. If we restrict ourselves to multilinear elements, being closed under endomorphisms is easy to translate into the language of operads: this means that we place ourselves in the category of right modules over an operad. By contrast, being an ideal does not correspond to considering left modules: the composition product $\circ$ is highly nonlinear, and so a left module structure corresponds to substituting an element of the left module into each slot of an operation from an operad. For an ideal in an operad, we may substitute an element of an ideal just in one slot of an operation of an operad, and this corresponds to the notion of an \emph{infinitesimal left module}. Altogether, an ideal of an operad $\calO$ has two commuting structures: that of a right $\calO$-module and that of an infinitesimal left $\calO$-module, or, using the existing terminology, of an \emph{infinitesimal $\calO$-bimodule}. This is a notion that will be useful for us below.

The free symmetric operad generated by a linear species $\calX$ is defined as follows. Its underlying linear species is the species $\calT(\calX)$ for which $\calT(\calX)(I)$ is spanned by decorated rooted trees (including the rooted tree without internal vertices and with just one leaf, which corresponds to the unit of the operad): the leaves of a tree must be in bijection with $I$, and each internal vertex $v$ of a tree must be decorated by an element of $\calX(I_v)$, where $I_v$ is the set of incoming edges of $v$. Such decorated trees should be thought of as tensors: they are linear in each vertex decoration. The operad structure is given by grafting of trees onto each other. (If one prefers the skeletal definition, one can talk about the free operad generated by a collection of $S_n$-modules, but the formulas will become heavier.) Free operads can be used to present operads by generators and relations; from the point of view of varieties, a presentation of an operad by generators and relations is extremely natural: the generators correspond to the signature (all structure operations of a variety, with their symmetries already taken into account), and the relations correspond to identities satisfied by the structure operations. 

\subsection{Operads of commutative associative differential algebras and  of Novikov algebras}

Commutative associative differential algebras form a variety and so may be encoded by an operad (since we work over a field of zero characteristic), which we shall denote $\CDi$. Specifically, the operad $\CDi$ is generated by a binary symmetric operation $\mu$ and a unary operation~$\delta$, and its ideal of relations is generated by the elements 
 \[
\mu(\mu(a_1,a_2),a_3)-\mu(a_1,\mu(a_2,a_3)),\quad  \delta(\mu(a_1,a_2))-\mu(\delta(a_1),a_2)-\mu(a_1,\delta(a_2)). 
 \]
We shall write $ab$ instead of $\mu(a,b)$, $a'$ instead of $\delta(a)$, and $a^{(s)}$ instead of $\delta^s(a)$ with $s\ge 2$, and also avoid unnecessary brackets where one may use the associativity of $\mu$; thus, for instance, we write $(a_1a_3'a_4)^{(2)}a_2$ instead of the rather bulky
 \[
\mu(\delta^2(\mu(\mu(a_1,\delta(a_3)),a_4)),a_2).
 \]
The relation
 \[
\delta(\mu(a_1,a_2))=\mu(\delta(a_1),a_2)+\mu(a_1,\delta(a_2)) 
 \]
is immediately seen to be a distributive law \cite[Sec.~8.6]{MR2954392} between the algebra $\k[\delta]$ (viewed as an operad supported in arity $1$) and the operad $\Com$ of commutative associative algebras. Consequently, on the level of underlying linear species, we have
$\CDi\cong\Com\circ \k[\delta]$, 
which means that as a basis of $\CDi(n)$ one may take all \emph{differential monomials} 
 \[
a_1^{(i_1)}a_2^{(i_2)}\cdots a_n^{(i_n)}. 
 \]
This, leads to a description of free differential polynomial algebras that matches the classical description of that algebra, see, e.g. \cite{Kolchin,Ritt}.

The operad $\Nov$ encoding Novikov algebras is generated by one binary operation $\nu$, and its ideal of relations is generated by the elements
\begin{gather*}
\nu(\nu(a_1,a_2),a_3)-\nu(a_1,\nu(a_2,a_3)-
\nu(\nu(a_2,a_1),a_3)+\nu(a_2,\nu(a_1,a_3),\\
\nu(\nu(a_1,a_2),a_3)-\nu(\nu(a_1,a_3),a_2).
\end{gather*}
We shall write $a\circ b$ instead of $\nu(a,b)$. 

In the introduction, we discussed the functor
 \[
\CDi\text{-alg} \to \Nov\text{-alg}
 \]
assigning to a commutative associative differential algebra structure on a vector space the Novikov algebra structure on the same vector space given by the product $x\circ y=xD(y)$. Operadically, this functor corresponds to the morphism of operads 
 \[
\imath\colon \Nov\to \CDi 
 \]
defined on the generator $\nu$ of $\Nov$ by $\imath(\nu)=\mu\circ_2\delta$. 

\begin{proposition}\label{prop:DLbasis}
The morphism $\imath$ is injective. Its image coincides with the linear span of all differential monomials $a_1^{(i_1)}a_2^{(i_2)}\cdots a_n^{(i_n)}$ for which $i_1+\cdots+i_n=n-1$.
\end{proposition}

\begin{proof}
Since the arity $n$ component of an operad may be identified with the space of multilinear elements of the $n$-generated free algebras, this result follows immediately of the corresponding statement about free algebras established by Dzhumadildaev and L\"ofwall \cite[Th.~7.8]{DzhL}.
\end{proof}

This result is of key importance to us: it means that we may perform various calculations inside the operad $\CDi$ whose basis and structure constants are much more intuitive than those of the operad $\Nov$. 

\section{Identities of differential algebras}\label{sec:toy}

To warm up for the proof of our main result, we shall first consider a toy model of commutative associative differential algebras. 

\subsection{Nilpotence of the image of the derivation}
Our first result asserts that for a commutative associative differential algebra satisfying nontrivial identities, the ideal generated by the image of $D$ is nilpotent. In operadic terms, this can be stated precisely as follows.  

\begin{theorem}\label{t2} Let $f$ be a nonzero element of the operad $\CDi$. If $\mathrm{char}(\k)=0$, the ideal of the operad $\CDi$ generated by $f$ contains an element of the form $a_1'\cdots a_m'$ for some $m\geq 1$.
\end{theorem}

\begin{proof}
Suppose that $f\in \CDi(k)$. We may write  
 \[
f=g_n a_k^{(n)}+g_{n-1}a_k^{(n-1)}+\ldots+g_1a_k'+g_{0}a_k,
 \]
where $g_i\in \CDi(k-1)$ and $g_n\ne 0$. (We suppress the arguments $a_1,\ldots,a_{k-1}$ that will remain unchanged throughout a long calculation that we shall perform.) It is clear that
 \[
h_0(a_1,\ldots,a_{k-1},a_k,a_{k+1}):=f(a_1,\ldots,a_{k-1},a_k)a_{k+1}-f(a_1,\ldots,a_{k-1},a_{k+1})a_k \]
is an element of the operadic ideal generated by $f$, and we have
 \[
h_0(a_1,\ldots,a_{k+1})=\sum_{s=1}^ng_s(a_1,\ldots,a_{k-1})(a_k^{(s)}a_{k+1}-a_{k+1}^{(s)}a_k).
 \]
In the calculation that follows, we shall denote for brevity $a_k=y$, $a_{k+1}=z$, $a_{k+2}=t$. We observe that $h_0$ is a linear combination of $y^{(s)}z-z^{(s)}y$ with $1\leq s\leq n$, and the term with $s=n$ appears with the coefficient $g_n$. 
Let us now define 
 \[
h_1:=h_0(yt,z)-h_0(y,z)t-h_0(zt,y)+h_0(z,y)t. 
 \]
Since we have 
\begin{multline*}
h_0(yt,z)-h_0(y,z)t=\\ \sum_{s=1}^ng_s\left(\sum_{s=0}^s\binom{s}{i}y^{(s-i)}t^{(i)}z-z^{(s)}yt-(y^{(s)}z-z^{(s)}y)t\right)=
\sum_{s=1}^ng_s\left(\sum_{s=1}^s\binom{s}{i}y^{(s-i)}t^{(i)}z\right),
\end{multline*}
we may write
\begin{multline*}
h_1(y,z,t)=h_0(yt,z)-h_0(y,z)t-h_0(zt,y)+h_0(z,y)t=\\ 
\sum_{s=1}^ng_s\left(\sum_{i=1}^s\binom{s}{i}y^{(s-i)}t^{(i)}z-\sum_{i=1}^s\binom{s}{i}z^{(s-i)}t^{(i)}y\right)=
\sum_{s=1}^ng_s\sum_{i=1}^s\binom{s}{i}t^{(i)}(y^{(s-i)}z-z^{(s-i}y)=\\
\sum_{s=1}^{n-1}\sum_{i=1}^{n-s}g_{s+i}\binom{s+i}{i}t^{(i)}(y^{(s)}z-z^{(s)}y).
\end{multline*}

We observe that the polynomial $h_1(y,z,t)$ is also a linear combination of the polynomials $y^{(s)}z-z^{(s)}y$, but we now have $1\leq s\leq n-1$, and the term with $s=n-1$ appears with the coefficient $ng_nt'$. Iterating this procedure, we may go all the way to the polynomial $h_{n-1}(y,z,t_1,\ldots,t_{n-1})=n!g_nt_1'\cdots t_{n-1}'(y'z-yz')$. Finally, 
 \[
h_{n-1}(yt_n,z,t_1,\ldots,t_{n-1})-h_{n-1}(y,z,t_1,\ldots,t_{n-1})t_n=n!g_nt_1'\cdots t_{n-1}'t_n'yz.
 \]
Since $g_n\in\CDi(k-1)$, we may proceed by induction on arity, concluding that the ideal generated by $f$ contains a polynomial of the form $a_1'\cdots a_p' a_{p+1}\cdots a_{p+q}$ for some $p,q$. It remains to use the substitutions $a_{p+i}\mapsto a_{p+i}'$ to transform this into $a_1'\cdots a_m'$ with $m=p+q$, completing the proof.
\end{proof}

This result has the following amusing corollary which one might view as a counterpart of the celebrated theorem of Kaplansky on primitive PI-rings~\cite{MR25451}. 

\begin{corollary}\label{cor:Kapl}
Suppose that $K\colon\k$ is a field extension of zero characteristic, and that $D\colon K\to K$ is a nonzero $\k$-linear derivation. Then $K$ does not satisfy any differential identity.
\end{corollary}

We were informed by Pogudin that this corollary, and even its version for several commuting derivations appears in the work of Kolchin \cite[Sec.~3]{MR7011} on primitive elements in differential fields. Pogudin also informed us that in his own PhD thesis, a stronger version of Corollary \ref{cor:Kapl} is established (by different methods): a prime differential algebra with a nonzero derivation does not satisfy any differential identity  \cite[Prop.~2.2.3]{Pogudin}. 

\subsection{Specht property for differential identities}\label{sec:dif-Specht}

We shall now establish that every system of multilinear differential identities admits a finite basis. In operadic terms, we prove the following result.

\begin{theorem}\label{th:CDSpecht}
Every ideal of the operad $\CDi$ is finitely generated.
\end{theorem}

\begin{proof}
Let us assign to the basis element $e=a_1^{(i_1)}\cdots a_{n}^{(i_{n})}\in\CDi(n)$ the element $\alpha(e)\in \mathbb{N}^n$, given by the sequence $(i_1,\ldots,i_n)$ of all differential degrees appearing in $e$, in order of appearance. Using this assignment, we may define an ordering of basis elements of the same arity $n$: we say that $e\preceq e'$ if $\alpha(e)\le\alpha(e')$ lexicographically. 

An obvious property of the ordering that we just introduced is the following compatibility with the product and the derivation: if $e=a_1^{(i_1)}\cdots a_{n}^{(i_{n})}$ is the leading term of an element $f\in \CDi$, then
\begin{itemize}
\item  the leading term of $f(a_1,\ldots, a_{k-1}, a_k',a_{k+1},\ldots,a_n)$ is 
 \[
e(a_1,\ldots, a_{k-1}, a_k',a_{k+1},\ldots,a_n)=a_1^{(i_1)}\cdots a_{k-1}^{(i_{k-1})} a_k^{(i_k+1)}a_{k+1}^{(i_{k+1})}\cdots a_{n}^{(i_{n})},
 \]
\item the leading term of $f(a_1,\ldots,a_{k-1},a_{k+1},\ldots,a_{n+1})a_k$ is 
 \[
e(a_1,\ldots,a_{k-1},a_{k+1},\ldots,a_{n+1})a_k 
=a_1^{(i_1)}\cdots a_{k-1}^{(i_{k-1})}a_k a_{k+1}^{(i_{k})}\cdots a_{n+1}^{(i_{n})}.
 \]
\end{itemize}

Suppose that $I$ is an ideal of $\CDi$ that is not finitely generated. Then we can define an infinite sequence of elements $\{f_k\}_{k\in\mathbb{N}}$ by choosing as $f_k$ the element of~$I$ with the smallest possible leading term $\hat{f_k}$ (with respect to the order $\preceq$) that does not belong to the ideal of $\CDi$ generated by $f_1,\ldots, f_{k-1}$. We shall assume that for each $k$ the leading coefficient of $f_k$ is equal to $1$.

Let us consider an exotic partial order on $\mathbb{N}$ for which the positive elements are ordered in the usual way, and the element $0$ is incomparable with any other elements. We extend this partial order to a partial order of the set $\mathbb{N}^*$ of all finite sequences of elements of $\mathbb{N}$ by saying that $\alpha\trianglelefteq\beta$ if $\alpha$ is majorized by a subsequence of $\beta$. By a result of Higman \cite[Th.~4.3]{MR49867}, $(\mathbb{N}^*,\trianglelefteq)$ is a partial well order. This implies that for any sequence $\{\alpha_s\}$ of elements of $\mathbb{N}^*$, there exists $s<t$ such that $\alpha_s\trianglelefteq \alpha_t$. If we define $\alpha_s:=\alpha(f_s)$, this means that one can find $s<t$ such that there exists a subsequence of the sequence of arguments of the monomial~$\hat{f_t}$ such that the arguments of differential degree $0$, as well as of strictly positive differential degrees in that subsequence are in one-to-one correspondence with the arguments of $\hat{f_s}$, and, moreover, for each argument of strictly positive differential degree, the differential degree in $\hat{f_t}$ is greater than or equal to the corresponding differential degree in $\hat{f_s}$. We shall refer to this by simply saying that $\hat{f_s}$ divides $\hat{f_t}$.

Let us show that $f_t$ belongs to the ideal generated by $f_s$. Applying several multiplications by variables of differential degree $0$, we may replace $f_s$ by an element $g$ from the ideal generated by $f_s$ such that $\hat{g}$ divides $\hat{f_t}$ and that $\hat{g}$ and $\hat{f_t}$ have the same amount of arguments of differential degree $0$. Furthermore, using the substitutions $a_k\mapsto a_k'$, we may replace $g$ by an element $h$ from the ideal generated by $f_s$ such that $\hat{h}$ divides $\hat{f_t}$ and the sequences of strictly positive differential degrees in $\hat{h}$ and $\hat{f_t}$ are the same, which forces $\hat{h}=\hat{f_t}$. The difference $f_t-h$ has the leading term smaller than that of $f_t$, and hence belongs to the ideal generated by $f_1,\ldots, f_{t-1}$. Since $h$ belongs to the ideal generated by $f_s$, we conclude that $f_t=(f_t-h)+h$ belongs to the ideal generated by $f_1,\ldots, f_{t-1}$, which is a contradiction. It follows that the ideal $I$ is finitely generated.
\end{proof}

\begin{corollary}
Over a field of zero characteristic, every system of differential identities is equivalent to a finite one.
\end{corollary}

\begin{proof}
In that case, every identity is equivalent to a multilinear one, so Theorem~\ref{th:CDSpecht} applies. 
\end{proof}

\section{Identities of Novikov algebras}\label{sec:IdNov}

In this section, we prove all the main results of the article: various nilpotence and solvability properties for Novikov algebras satistying nontrivial identities, as well as the Specht property for multilinear Novikov identities.

\subsection{Monomials in infinitesimal \texorpdfstring{$\Nov$}{Nov}-bimodules}

In this section, we prove an auxiliary result that we shall later use to study Novikov algebras satistying nontrivial identities. Considering the Novikov operad as a suboperad of the operad $\CDi$ allows us to view the operad $\CDi$ as an infinitesimal bimodule over the operad $\Nov$. Considering sub-bimodules of this infinitesimal bimodule is very natural; it is somewhat similar to considering weak identities \cite{MR4285755}, but is more restrictive, since we only allow substitutions of Novikov elements and multiplying by a Novikov element using a Novikov product (only the first constraint would appear for a weak identity). In this context, we prove the following result resembling Theorem~\ref{t2}. 

\begin{theorem}\label{th:mon-bimod}
Let $f$ be a nonzero element of the operad $\CDi$. If $\mathrm{char}(\k)=0$, the infinitesimal $\Nov$-bimodule of the operad $\CDi$ generated by $f$ contains an element of the form
 \[
a_1^{(2)}\cdots a_p^{(2)} a_{p+1}\cdots a_{p+q}
 \]
for some $p,q\geq 1$.
\end{theorem}

\begin{proof}
Suppose that $f\in \CDi(k)$. We may write  
 \[
f=g_n a_k^{(n)}+g_{n-1}a_k^{(n-1)}+\ldots+g_1a_k'+g_{0}a_k,
 \]
where $g_i\in \CDi(k-1)$ and $g_n\ne 0$. (We suppress the arguments $a_1,\ldots,a_{k-1}$ that will remain unchanged throughout a long calculation that we shall perform.)

Suppose that $n>2$. Let us consider the element 
\[
h_0(a_k,a_{k+1})=f(a_ka_{k+1}')-f(a_k)a_{k+1}'-f(a_k'a_{k+1})+a_k'f(a_{k+1})
\]
of the infinitesimal $\Nov$-bimodule generated by $f$. It is clear that 
\[
h_0(a_k,a_{k+1}) =\sum_{s=1}^ng_s\sum_{p=1}^s\binom{s}{p}(a_k^{(s-p)}a_{k+1}^{(p+1)}-a_k^{(p+1)}a_{k+1}^{(s-p)}).
\]
 
Let us denote for brevity $f_{p,q}(a_k,a_{k+1})=a_k^{(p)}a_{k+1}^{(q)}-a_k^{(q)}a_{k+1}^{(p)}$. Since $f_{p,q}=-f_{q,p}$, some terms in the formula for $h_0$ are proportional. Collecting the proportional terms, we obtain the formula
 \[
h_0(a_k,a_{k+1})=\sum_{s=1}^ng_s\left(f_{0,s+1}+sf_{1,s}+\sum_{2\leq p<\lfloor\frac{s+1}{2}\rfloor}
\left(\binom{s}{p}-\binom{s}{p-1}\right)f_{p,s+1-p}
\right).
 \]
In the calculation that follows, we shall denote for brevity $a_k=y$, $a_{k+1}=z$, $a_{k+2}=t$.

We observe that $h_0$ is a linear combination of $f_{p,q}(y,z)$ with $p<q$, $p+q=2,\ldots,n+1$, and each such element with $p+q=n+1$ appears with the coefficient $\alpha_{p,q} g_n$, where $\alpha_{p,q}$ is a positive integer. Let us now define
 \[
h_1:=h_0(yt',z)-h_0(y,z)t'-h_0(zt',y)+h_0(z,y)t', 
 \]
which clearly belongs to the infinitesimal $\Nov$-bimodule generated by $h_0$, and hence to the infinitesimal $\Nov$-bimodule generated by $f$.
Since we have 
 \[
f_{p,q}(yt',z)-f_{p,q}(y,z)t'=
\sum_{i=1}^p\binom{p}{i}y^{(p-i)}t^{(i+1)}z^{(q)}-\sum_{j=1}^l\binom{q}{j}y^{(q-j)}t^{(j+1)}z^{(p)},
 \]
therefore
\begin{multline*}
f_{p,q}(yt',z)-f_{p,q}(y,z)t'-f_{p,q}(zt',y)+f_{p,q}(z,y)t'=\\
\sum_{i=1}^p\binom{p}{i}t^{(i+1)}(y^{(p-i)}z^{(q)}-y^{(q)}z^{(p-i)})-\sum_{j=1}^l\binom{q}{j}t^{(j+1)}(y^{(q-j)}z^{(p)}-y^{(p)}z^{(q-j)})=\\
\sum_{i=1}^p\binom{p}{i}t^{(i+1)}f_{p-i,q}(y,z)+\sum_{j=1}^q\binom{q}{j}t^{(j+1)}f_{p,q-j}(y,z).
\end{multline*}
It follows that $f_{p,q}(yt',z)-f_{p,q}(y,z)t'-f_{p,q}(zt',y)+f_{p,q}(z,y)t'$ is a linear combination of the polynomials $f_{a,b}(y,z)$ for $a+b=p+q-j$ for some $j\geq 1$. In particular, if $p<q$ and we focus on $j=1$, we get two terms  
 \[
kt''f_{k-1,l}(y,z)+ lt''f_{k,l-1}(y,z),
 \]
where the second polynomial might vanish (if $k=l-1$), but both coefficients are non-negative.

Collecting our observations, we see that $h_1$ is a linear combination of the polynomials $f_{p,q}(y,z)$ with $p<q$ and $p+q=1,\ldots,n$, and such polynomials with $p+q=n$ appear with non-negative integer coefficients not all of which are zeros. Iterating this procedure, we may go all the way to the polynomial $h_{n-1}=\lambda_ng_nt_1''\cdots t_{n-1}''(y'z-yz')$, where $\lambda_n$ is a nonzero integer. 
Finally, we consider the element 
 \[
h_n:=h_{n-1}(yt_n',z)-h_{n-1}(y,z)t_n'=\lambda_ng_nt_1''\cdots t_{n-1}''t_n''yz.
 \]
Since $g_n\in\CDi(k-1)$, we may proceed by induction on arity, obtaining an element of the required form 
 \[
a_1^{(2)}\cdots a_p^{(2)} a_{p+1}\cdots a_{p+q}
 \]
for some $p,q\geq 1$.
\end{proof}

\subsection{Iterated associators in ideals of the Novikov operad}

In this section, we apply Theorem \ref{th:mon-bimod} to study Novikov algebras satisfying nontrivial identities. 

\begin{theorem}\label{th:mon-nov}
Let $N$ be a Novikov algebra over a field of zero characteristic satisfying a nontrivial identity. Then $N$ is right associator nilpotent. 
\end{theorem} 

\begin{proof}
Note that under the embedding of  operads $\Nov\hookrightarrow\CDi$, the associator $(a_1\circ a_2)\circ a_3-a_1\circ(a_2\circ a_3)$ is sent to
 \[
(a_1a_2')a_3'-a_1(a_2a_3')'=-a_1a_2a_3^{(2)}.
 \]
This immediately implies that the image of the element 
 \[
((\ldots ((a_1,a_2,a_3),a_4,a_5),\ldots,),a_{2p},a_{2p+1}) 
 \]
is precisely the differential monomial $a_1a_2a_3^{(2)}a_4a_5^{(2)}\cdots a_{2p}a_{2p+1}^{(2)}$. Thus, we must show that if $f$ is a nonzero element of operad $\Nov$, then under the embedding of  operads $\Nov\hookrightarrow\CDi$, the ideal of $\Nov$ generated by $f$ contains a differential monomial of that form.

It follows from Theorem \ref{th:mon-bimod} that the infinitesimal $\Nov$-bimodule generated by $f$ in the operad $\CDi$ contains an element of the form $a_1^{(2)}\cdots a_p^{(2)} a_{p+1}\cdots a_{p+q}$ for some $p,q\geq 1$. Since $f\in\Nov$, we have
 \[
a_1^{(2)}\cdots a_p^{(2)} a_{p+1}\cdots a_{p+q}\in \Nov,
 \]
which can only happen if $2p=p+q-1$, so $q=p+1$, and, up to a permutation of arguments, this is an element of required form.
\end{proof}

Let us obtain several corollaries of this result. First, we record the following nilpotence result for the commutator $[a,b]=a\circ b-b\circ a$ in Novikov algebras. 

\begin{corollary}\label{cor:com-nil}
Let $N$ be a Novikov algebra over a field of zero characteristic satisfying a nontrivial identity. For some  $n\ge 1$, the identity
 \[
(\ldots((a_1\circ[a_2,a_3])\circ[a_4,a_5])\ldots)\circ[a_{2n},a_{2n+1}]=0 
 \]
holds in~$N$. 
\end{corollary}

\begin{proof}
Note that under the embedding of  operads $\Nov\hookrightarrow\CDi$, the element $a_1\circ(a_2\circ a_3-a_3\circ a_2)$ is sent to
 \[
a_1(a_2a_3'-a_3a_2')'=a_1a_2a_3^{(2)}-a_1a_2^{(2)}a_3.
 \]
This means that in iterated right products of commutators second derivatives accummulate, and one can argue like in the proof of Theorem \ref{th:mon-nov}.
\end{proof}

For the next corollary, we recall that in every Novikov algebra $N$ over a field of characteristic different from~$2$, the space $[N,N]$ spanned by all commutators is an ideal \cite{MR2419136,TUZ21}.

\begin{corollary}\label{cor:com-right-nil}
Let $N$ be a Novikov algebra over a field of zero characteristic satisfying a nontrivial identity. The commutator ideal $[N,N]$ is right nilpotent (as a Novikov algebra). 
\end{corollary}

\begin{proof}
This is an immediate consequence of Corollary \ref{cor:com-nil}.
\end{proof}

Theorem \ref{th:mon-nov} and Corollary \ref{cor:com-nil} admit the following common generalization, which is proved in the exact same way.

\begin{corollary}\label{cor:as-com-nil}
Let $N$ be a Novikov algebra over a field of zero characteristic satisfying a nontrivial identity. The subalgebra of $\mathrm{End}(N)$ generated by all the endomorphisms of the form $x\mapsto (x,y,z)$ and $x\mapsto x[y,z]$ is nilpotent. 
\end{corollary}

Our results have the following appealing corollary.

\begin{corollary}\label{cor:lie-sol}
Over a field of zero characteristic, a Novikov algebra $N$ satisfies a nontrivial identity if and only if it is Lie-solvable.
\end{corollary}

\begin{proof}
For the Novikov algebra $\k[x]$ equipped with the product $f\circ g=fg'$, the corresponding Lie algebra is the \emph{simple} Lie algebra $W_1$ of vector fields on the affine line, so the property of being Lie-solvable of certain degree is a nontrivial Novikov identity. To prove the reverse implication, we note that according to Corollary \ref{cor:com-right-nil}, the commutator ideal $[N,N]$ is right nilpotent, therefore, according to a result of Shestakov and Zhang, the Novikov algebra $[N,N]$ is solvable, and therefore $[N,N]$ is Lie-solvable.
\end{proof}

Corollary \ref{cor:lie-sol} leads to an alternative proof of the following result established by Burde and Dekimpe in \cite{MR2240426}. 

\begin{corollary}
Over a field of zero characteristic, a finite-dimensional Lie algebra admitting a Novikov structure is solvable.
\end{corollary}

\begin{proof}
It is known \cite{MR3241181} that free Novikov algebras over a field of zero characteristic contain polynomial representations corresponding to Young diagrams with arbitrary many rows. It follows that for each $d$ there exist a nontrivial multilinear identity that is skew-symmetric in $d$ of its arguments, so every finite-dimensional Novikov algebra satisfies a nontrivial identity, and Corollary \ref{cor:lie-sol} applies.
\end{proof}

Finally, we record an alternative proof of the following statement that was first established by Makar-Limanov and the third author in \cite{MLU11N}.

\begin{corollary}\label{cor-W1}
If $\mathrm{char}(\k)=0$, the Novikov algebra $\k[x]$ equipped with the product $f\circ g=fg'$ does not satisfy any nontrivial Novikov identity. 
\end{corollary}

\begin{proof}
Clearly, no differential identity of the form $a_1^{(2)}\cdots a_p^{(2)} a_{p+1}\cdots a_{2p+1}=0$ holds in the polynomial algebra.
\end{proof}

\subsection{Specht property for Novikov identities}

In this section, we shall prove that every system of multilinear identities of Novikov algebras admits a finite basis. In operadic terms, we prove the following result.

\begin{theorem}\label{th:NovSpecht}
Every ideal of the operad $\Nov$ is finitely generated.
\end{theorem}

\begin{proof}
Let us define an ordering of basis elements of the same arity of the operad $\Nov$ as follows. In addition to the sequence $\alpha(e)$ of differential degrees appearing in $e$, we assign to each basis element $e$ of arity $n$ the parameter $n(e)$ equal to the number of arguments of $e$ whose differential degree is equal to zero. For two basis elements $e$ and $e'$ of arity $n$, we say that $e\preceq e'$ if $(n(e),\alpha(e))<(n(e'),\alpha(e'))$ lexicographically. 

An almost obvious property of the ordering that we just introduced is the following compatibility with the Novikov substitutions: if $e=a_1^{(i_1)}\cdots a_{n}^{(i_{n})}$ is the leading term of an element $f\in \Nov$, then for each total order on the arguments $a_1,\ldots,a_n,b$ that extends the standard order on $a_1,\ldots,a_n$ (an order of arguments is needed to form the \emph{sequence} of differential degrees), the leading term of the element $f(a_1,\ldots,a_{k-1}, a_k'b,a_{k+1},\ldots,a_n)$ is 
 \[
a_1^{(i_1)}\cdots a_{k-1}^{(i_{k-1})} a_k^{(i_k+1)}a_{k+1}^{(i_{k+1})}\cdots a_{n}^{(i_{n})} b,
 \]
where in the latter formula it is implied that the factor $b$ is placed in its correct place according to the given order of arguments. To prove this assertion, it is enough to note that for $s>0$, then all terms except for the last term $a^{(s+1)}b$ in the formula
 \[
(a'b)^{(s)}=\sum_{i=0}^{s} \binom{s}{i}a^{(i+1)}b^{(s-i)}
 \]
are products of two elements of positive differential degrees, so $a^{(s+1)}b$ will the leading term of this Novikov substitution according to our order of the basis. 

Suppose that $I$ is an ideal of $\Nov$ that is not finitely generated. Then we can define an infinite sequence of elements $\{f_k\}_{k\in\mathbb{N}}$ by choosing as $f_k$ the element of~$I$ with the smallest possible leading term $\hat{f_k}$ (with respect to the order $\preceq$) that does not belong to the ideal of $\Nov$ generated by $f_1,\ldots, f_{k-1}$. We shall assume that for each $k$ the leading coefficient of $f_k$ is equal to $1$.

We shall now use an argument similar to that of the proof of Theorem \ref{th:CDSpecht}. In that proof, we established that one can find $s<t$ such that there exists a subsequence of the sequence of arguments of the monomial~$\hat{f_t}$ such that the arguments of differential degree $0$, as well as of strictly positive differential degrees in that subsequence are in one-to-one correspondence with the arguments of $\hat{f_s}$, and, moreover, for each argument of strictly positive differential degree, the differential degree in $\hat{f_t}$ is greater than or equal to the corresponding differential degree in $\hat{f_s}$. As before, we shall refer to this by simply saying that $\hat{f_s}$ divides $\hat{f_t}$.

Let us show that $f_t$ belongs to the ideal generated by $f_s$. Using the above observation about the leading terms, we note that we can apply several Novikov substitutions to replace $f_s$ by an element $g$ from the ideal generated by $f_s$ such that $\hat{g}$ divides $\hat{f_t}$ and that the sequences of strictly positive differential degrees in $\hat{g}$ and $\hat{f_t}$ are the same. Since the arity and the differential degree of a Novikov element differ by one, this forces $\hat{g}=\hat{f_t}$.  The difference $f_t-g$ has the leading term smaller than that of $f_t$, and hence belongs to the ideal generated by $f_1,\ldots, f_{t-1}$. Since $g$ belongs to the ideal generated by $f_s$, we conclude that $f_t=(f_t-g)+g$ belongs to the ideal generated by $f_1,\ldots, f_{t-1}$, which is a contradiction. It follows that the ideal $I$ is finitely generated.
\end{proof}

\begin{corollary}
Over a field of zero characteristic, every system of Novikov identities is equivalent to a finite one.
\end{corollary}

\begin{proof}
In that case, every identity is equivalent to a multilinear one, so Theorem~\ref{th:NovSpecht} applies. 
\end{proof}

Let us finish with noting that, by contrast with Theorems \ref{th:CDSpecht} and \ref{th:NovSpecht}, not every infinitesimal sub-$\Nov$-bimodule of $\CDi$ is finitely generated. As an example, we may consider the elements $a_1^{(k)}a_2^{(k)}$ for $k\geq 1$: since Novikov substitutions and multiplications by Novikov elements increase the arity, none of these elements may be contained in the infinitesimal $\Nov$-bimodule generated by the others. In fact, would not even help to allow the left infinitesimal action of the bigger operad $\CDi$, since the differential ideal in the free algebra in two generators generated by the elements $a_1^{(k)}a_2^{(k)}$ for $k\ge 0$ is not finitely generated. This latter statement is mentioned in \cite{MR3973134}, and is proved analogously to the result of Ritt \cite[p.~11]{Ritt} stating that the differential ideal generated by the elements $a_1^{(k)}a_2^{(l)}$ for $k,l\ge 0$ is not finitely generated.

\section*{Acknowledgments}

We thank Gleb Pogudin for enlightening comments about identities in differential algebras. The crucial part of work on this paper was prepared during the authors' stay at the Max Plank Institute for Mathematics in Bonn and they wish to express their gratitude to that institution for hospitality and excellent working conditions. This research is supported by the Institut Universitaire de France, by Fellowship USIAS-2021-061 of the University of Strasbourg Institute for Advanced Study through the French national program ``Investment for the future'' (IdEx-Unistra), by the French national research agency (project ANR-20-CE40-0016), and by the grants of the Ministry of Education and Science of the Republic of Kazakhstan (projects AP14872073 and AP14870282).

\bibliographystyle{plain}
\bibliography{biblio}

\end{document}